\documentclass{amsart}

\usepackage{graphicx}   
\usepackage{amssymb}    

\usepackage{prettyref}
\newrefformat{rem}{Remark \ref{#1}}
\newrefformat{cor}{Corollary \ref{#1}}
\newrefformat{fig}{Figure \ref{#1}} 

\newcommand{\nats}{\mathbb{N}}

\newcommand{\rsphere}{\widehat{\mathbb{C}}}

\newcommand{\set}[1]{\left\{#1\right\}}

\newtheorem{thm}{Theorem}
\newtheorem{lem}[thm]{Lemma} 
\newtheorem{cor}[thm]{Corollary} 

\theoremstyle{definition} 

\newtheorem{defn}[thm]{Definition} 
 
\newtheorem{rem}[thm]{Remark} 
\newtheorem{ques}[thm]{Question} 
\newtheorem*{term}{Terminology} 
\newtheorem*{ack}{Acknowledgements}

\begin{document}

\author[C. P.~Curry]{Clinton P. Curry} 

\title[Irreducible Julia Sets]{Irreducible Julia Sets of Rational
  Functions}

\address[Clinton P.~Curry] %
{Department of Mathematics\\
  University of Alabama at Birmingham\\
  Birmingham, AL 35294-1170} \email{clintonc@uab.edu}
\urladdr{http://www.math.uab.edu/~curry} \date{\today}

\thanks{The author was supported in part by NSF-DMS-0353825.}



\begin{abstract}
  We prove that a polynomial Julia set which is a finitely irreducible
  continuum is either an arc or an indecomposable continuum.  For the
  more general case of rational functions, we give a topological model
  for the dynamics when the Julia set is an irreducible continuum and
  all indecomposable subcontinua have empty interior.
\end{abstract}
\maketitle
\thispagestyle{empty}


\section{Introduction}
\label{sec:introduction}

Let $\rsphere$ denote the Riemann sphere, and let $R:\rsphere
\rightarrow \rsphere$ be a rational function.  The \emph{Fatou set}
$F(R)$ is the domain of normality of the iterates $\set{R^i \, \mid \, i
\in \nats}$.  The \emph{Julia set} $J(R)$ is $\rsphere \setminus
F(R)$, which is generally regarded as the subset of $\rsphere$ where
$R$ is chaotic.  If the degree of $R$ is at least two, then $J(R)$ is
a non-empty, compact, perfect subset of $\rsphere$.  A
\emph{continuum} is a non-empty, compact, and connected metric space.
A connected Julia set of a rational function is therefore a
subcontinuum of $\rsphere$.

This paper is motivated by recent work that addressed a conjecture of
P. M.~Ma\-ki\-en\-ko.  The exact statement is not important to the
subject of this paper; interested parties are referred to
\cite{Curry:ul} for more information.  What is important is the
conclusion: We proved that, if $R$ is a rational function that is a
counterexample to Makienko's conjecture, its Julia set $J(R)$ is a
\emph{finitely irreducible continuum}, and hence an
\emph{indecomposable continuum}.

\begin{defn}[Irreducible]
  A continuum $X$ is \emph{irreducible about a set $A \subset X$} if
  no proper subcontinuum of $X$ contains $A$.  If $X$ is irreducible
  about a finite subset, it is called \emph{finitely irreducible}.  If
  $X$ is irreducible about a two-point subset, it is simply called
  \emph{irreducible}.
\end{defn}
\begin{rem}\label{rem:lc_irred}
  The unit interval $[0,1]$ is an irreducible continuum, since no
  proper subcontinuum contains both of its endpoints. For the same
  reason, a finite tree is a finitely irreducible continuum.
  Conversely, any locally connected finitely irreducible continuum is a
  finite tree, since locally connected continua are also arcwise
  connected.  The arc is the only locally connected irreducible
  continuum.
\end{rem}
\begin{figure}
  \centering
  \includegraphics[height=4cm]{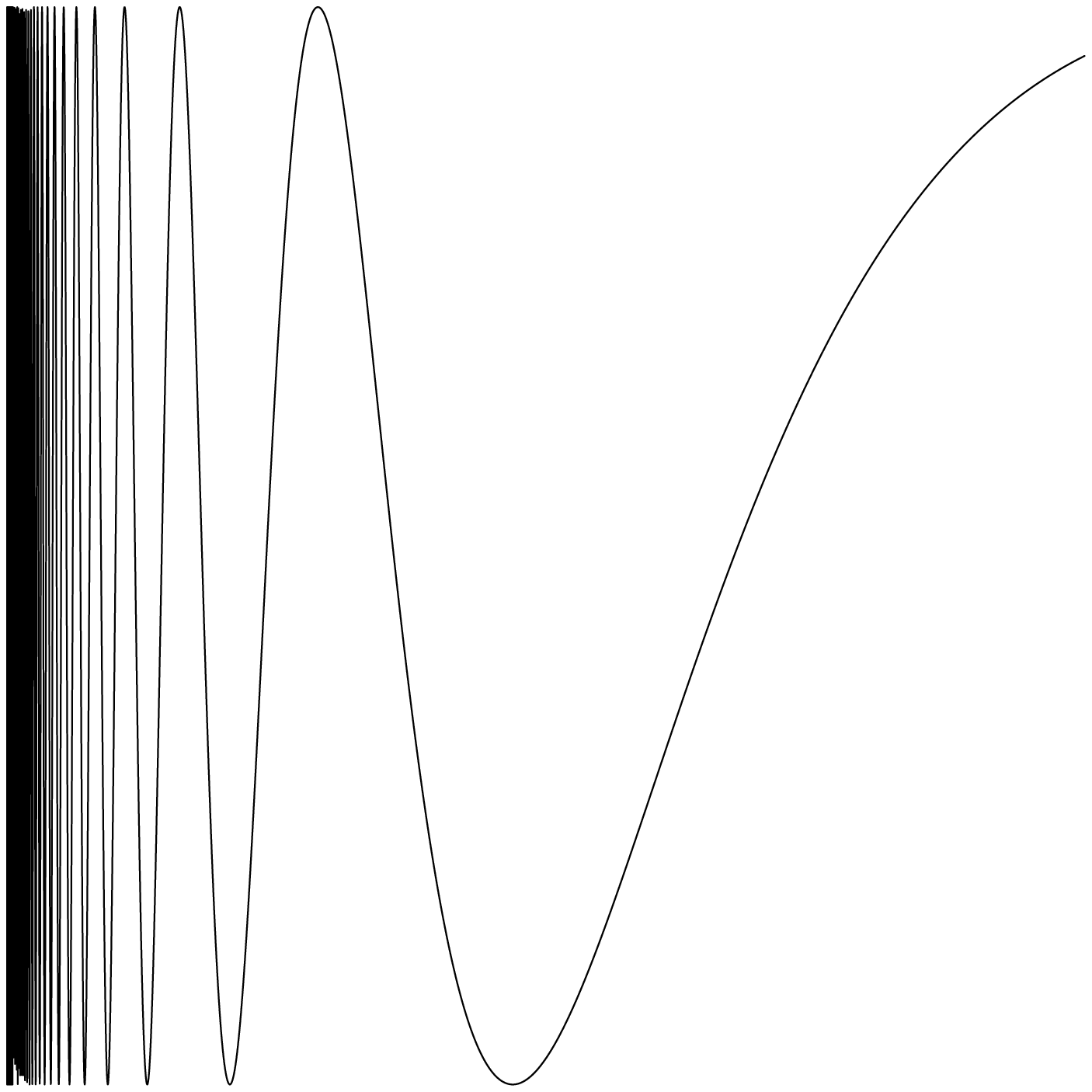}
  \hspace{1cm}
  \includegraphics[height=4cm]{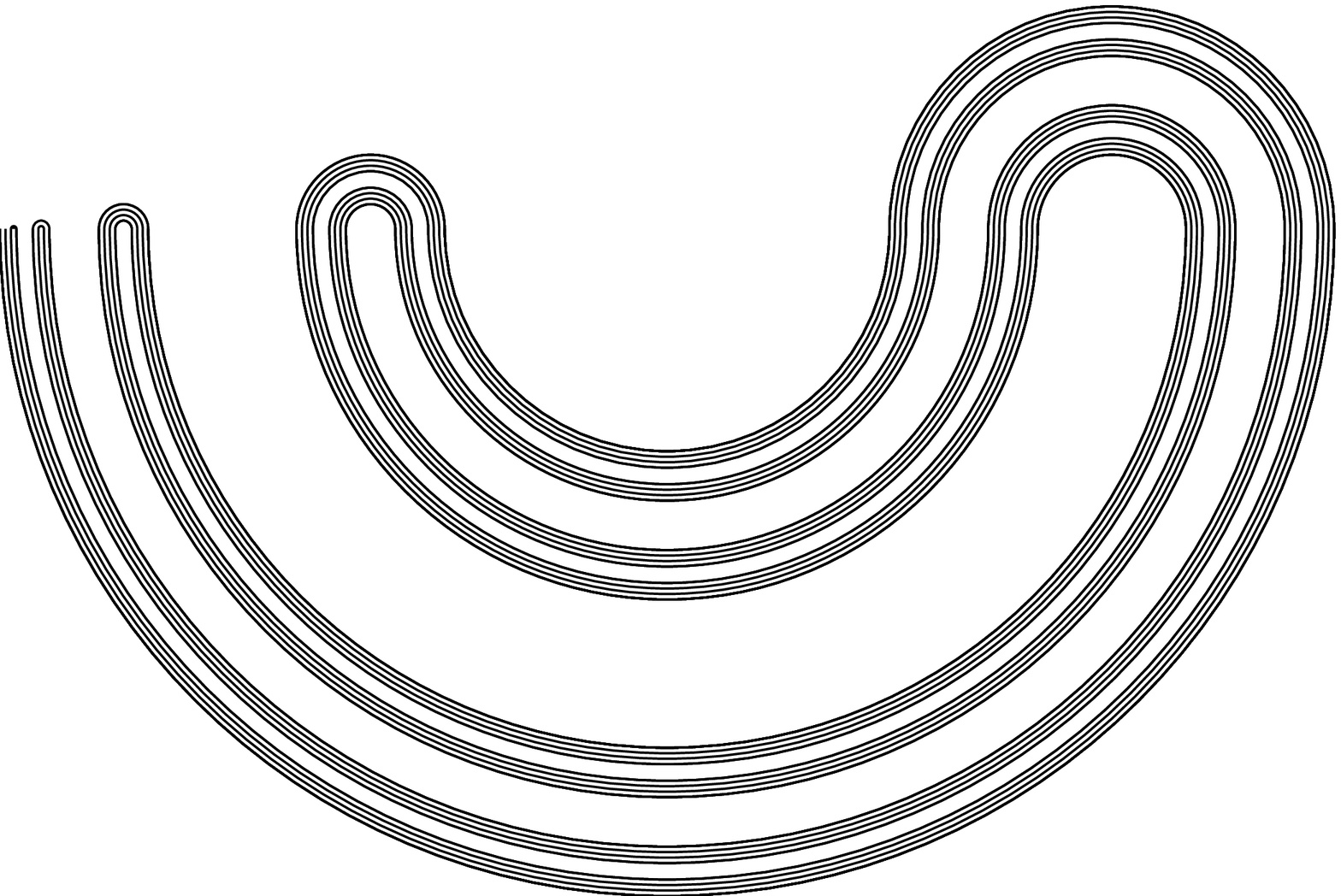}
  \caption{On the left is the $\sin \frac1x$ continuum, the simplest
    non-locally connected irreducible continuum.  On the right is the
    Knaster continuum, the simplest indecomposable continuum.}
  \label{fig:sin1x}
\end{figure}

It is well-known that there are Julia sets of rational functions which
are irreducible continua.  For example, the Julia set of the
polynomial $z \mapsto z^2-2$ is the interval $[-2,2]$.  Though it is
not known if every Julia set which is an irreducible continuum is an
arc, there are far more complicated examples of irreducible continua.

\begin{defn}[Indecomposable]
  A continuum $X$ is \emph{decomposable} if there are proper
  subcontinua $A$ and $B$ such that $A \cup B = X$.  Otherwise, $X$ is
  \emph{indecomposable}.
\end{defn}

\begin{rem}\label{rem:indec}
  Indecomposable continua are also irreducible continua, even strongly
  so.  If $X$ is an indecomposable continuum, then for a residual set
  of pairs $(x,y) \in X \times X$, $X$ is irreducible between $x$ and
  $y$.
\end{rem}

The arc and indecomposable continua represent the extremes of
topological complexity for irreducible continua.  The goal of this
paper is to prove that these two cases are representative for rational
Julia sets which are finitely irreducible continua.  We prove in
\prettyref{sec:polynomial-case} that if $P$ is a polynomial such that
$J(P)$ is an irreducible continuum, then $J(P)$ is either an arc or an
indecomposable continuum.  We conjecture the same result holds for
rational functions, but we only prove a weaker structure theorem for
rational functions using the same tools (see
\prettyref{sec:invariance}).

The argument naturally divides into two cases.  The first case is when
the Julia set in question is an irreducible continuum which has an
indecomposable subcontinuum with non-empty interior relative to the
Julia set.  Not much is said about the Julia set in this case, either
topologically or dynamically.  The other case is that the Julia set
contains no indecomposable subcontinuum with interior, which is the
case that most of the statements in this paper address.  To avoid
awkward repetition, we introduce the following terminology.

\begin{term}
  A continuum which is finitely irreducible such that no
  indecomposable subcontinuum has interior will be called a
  \emph{finished continuum}.
\end{term}

\begin{ack}
  I would like to express my appreciation for what Professor Devaney
  has done for the field and for the reserachers who work in it.  I
  would also like to thank my doctoral advisors, Dr. Alexander Blokh
  and Dr. John C.~Mayer, for helpful conversations on this and other
  topics.
\end{ack}

\section{Aposyndesis and Vought's Decomposition}
\label{sec:aposyndesis}

In this section, we review the notion of \emph{aposyndesis} due to
Jones \cite{Jones:1941db, Jones:1952nx} and an associated
decomposition defined by Vought \cite{Vought:1974qy}.  The goal is to
define a monotone map $m:J(R) \rightarrow Y$, where $Y$ is a
topologically simpler space.  In our case, it will happen that $Y$ is
locally connected, and studying $Y$ will help us gain insight about
$J(R)$.

\begin{defn}[Aposyndesis]
  A continuum $X$ is aposyndetic at a point $p \in X$ if every point
  $q \in X \setminus \set{p}$ is contained in the interior (relative to
  $X$) of a continuum $Y \subset X \setminus \set{p}$.
\end{defn}
Aposyndesis is a topological property which is slightly weaker than
local connectivity.  It is not difficult to see that any locally
connected continuum is aposyndetic.  Conversely, any aposyndetic
continuum which is the boundary of a simply connected domain in
$\rsphere$ is locally connected \cite[Theorem 14]{Whyburn:1939zr}.
Since the Julia set of a polynomial is the boundary of the domain of
attraction of $\infty$, an aposyndetic polynomial Julia set is also
locally connected.  A simple example of a non-locally connected
aposyndetic continuum is illustrated in \prettyref{fig:aposyndetic}.

Associated to the concept of aposyndesis is the set-valued function
$T$, defined below.  It should be considered in analogy with a
topological closure operator, except that it is not necessarily
idempotent.

\begin{figure}
  \centering
  \includegraphics[width=5cm]{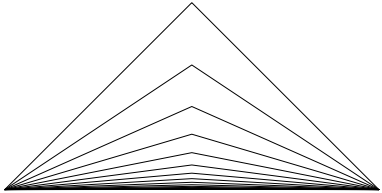}
  \caption{The suspension over a convergent sequence.  This continuum
    is aposyndetic, but not locally connected.}
  \label{fig:aposyndetic}
\end{figure}
\begin{defn}[The set-valued function $T$]
  For a non-empty set $A \subset X$, let $T(A)$ be the set of $x \in
  X$ for which all subcontinua of $X \setminus A$ containing $x$ have
  empty interior relative to $X$.  A set $A$ is called
  $\emph{$T$-closed}$ if $T(A) = A$.
\end{defn}
\begin{rem}\label{rem:T_properties}
  It follows immediately from the definition that, for any sets $A
  \subset B$, $A \subset T(A) \subset T(B)$.  If $X$ is the $\sin
  \frac 1 x$ continuum depicted in \prettyref{fig:sin1x} and $A$ is a subset
  of the limit bar, that $T(A)$ equals the entire limit bar.  Further,
  because every proper subcontinuum of an indecomposable continuum is
  nowhere dense, $T(A) = X$ whenever $X$ is an indecomposable
  continuum and $A$ is a subset.
\end{rem}

Vought showed that the set-valued function $T$ can be used to define a
continuous monotone map from any finished continuum $X$ onto a finite
tree $Y$.  He does so giving a partition $\mathcal D(X)$ of $X$ into
subcontinua. The space $Y$ is defined as the quotient space $X /
\mathcal D(X)$.  In general, $m$ is monotone when each element of
$\mathcal D(X)$ is connected, and $m$ is continuous when the
collection $\mathcal D(X)$ is upper semi-continuous.  The reader is
referred to \cite[Chapter 3]{Nadler:1992ek} for these and other facts
about decompositions of continua and continuous maps.

The following theorem describes the map defined by Vought.
\begin{thm}[Theorem 1 of \cite{Vought:1974qy}]\label{thm:decomp}
  Let $M$ be a finished continuum which is irreducible about $n$
  points, but no fewer.  Let
  \[\mathcal D(M) = \set{ T^n\{x\} : x \in M }\]
  where $T^n$ denotes the $n$-fold composition of $T$. Then the
  following hold:
  \begin{enumerate}
  \item $\mathcal D(M)$ is an upper semi-continuous decomposition of $M$,
  \item the elements of $\mathcal D(M)$ are continua,
  \item the quotient space $M / \mathcal D(M)$ is locally connected, and
  \item each element of $\mathcal D(M)$ has no interior in $M$.
  \end{enumerate}
  Further, $\mathcal D(M)$ is the only partition satisfying these
  properties.  Also, the quotient $M / \mathcal{D}(M)$ is a tree with
  $n$ endpoints.
\end{thm}

\begin{rem}\label{rem:trans_to_maps}
  This can be translated as the existence of a continuous monotone map
  $m:M \rightarrow Y$ where $Y$ is a finite tree with $n$ endpoints
  and point inverses are nowhere dense in $M$.
\end{rem}

For the remainder of the paper, when $X$ is a finished continuum the
symbol $\mathcal D(X)$ will denote the decomposition provided by this
theorem.  For a point $x \in X$, let $D(x)$ be the element of
$\mathcal D(X)$ containing $x$.  The following observation relates the
properties of $T$-closed sets to the decomposition $\mathcal D(X)$.

\begin{lem}\label{lem:closed_saturated}
  Let $X$ be a finished continuum irreducible about $n$ points.  Then
  every $T$-closed set is a union of elements of $\mathcal D(X)$.
\end{lem}
\begin{proof}
  Suppose that $A \subset X$ is $T$-closed and let $x \in A$.  By
  \prettyref{rem:T_properties}, $x \in T^n(\set{x}) \subset T^n(A)$,
  which equals $A$ since $A$ is $T$-closed.  But $T^n(\set{x}) \in \mathcal
  D(X)$ and $x \in A$ was arbitrary, so $A = \bigcup_{x \in A}D(x)$.
\end{proof}

\section{The Decomposition and the Dynamics}
\label{sec:invariance}
The goal of this section is to prove the following theorem.
\begin{thm}\label{thm:main_rat}
  Suppose that $R$ is a rational function whose Julia set is a
  finitely irreducible continuum.
  \begin{enumerate}
  \item $J(R)$ contains an indecomposable subcontinuum with non-empty
    interior relative to $J(R)$, or
  \item $J(R)$ is irreducible between two points, and admits a
    monotone map $m:J(R) \rightarrow[0,1]$ such that
    \begin{enumerate}
    \item $m^{-1}(x)$ is nowhere dense for all $x \in [0,1]$, and
    \item $m$ semiconjugates $R|_{J(R)}$ to an open, topologically
      exact, at most $\operatorname{deg}(R)$-to-one map
      $g:[0,1]\rightarrow[0,1]$.
    \end{enumerate}
  \end{enumerate}
\end{thm}

The main work is to show that the decomposition $\mathcal{D}(J(R))$,
defined in the previous section, respects the dynamics of $R$ when
$J(R)$ is a finished continuum.  Specifically, we aim to show that the
image under $R$ of any element of $\mathcal D(J(R))$ is again an
element of $\mathcal D(J(R))$.

First we prove two lemmas relating $T$ and $R$.
\begin{lem}\label{lem:preimage_apo}
  Let $R$ be a rational map.  If $A$ is $T$-closed in $J(R)$, then
  $R^{-1}(A)$ is as well.
\end{lem}
\begin{proof}
  Let $A \subset Y$ be a $T$-closed set, and let $A^{-1}=R^{-1}(A)$.
  Suppose $x \notin A^{-1}$; we will show that $x \notin T(A^{-1})$.
	
  Note that $R(x) \notin A$ since $x \notin T(A^{-1})\supset A^{-1}$.
  Since $A$ is $T$-closed, we have that $R(x) \notin T(A)$.  Let $H
  \subset Y \setminus A$ be a continuum containing $R(x)$ in its
  interior.  Then $R^{-1}(H)$ is a closed set that is also a
  neighborhood of $x$.  Let $H^{-1}$ be the component of $R^{-1}(H)$
  containing $x$.  Note that $R^{-1}(H)$ is the union of finitely many
  continua, so $H^{-1}$ is open in $R^{-1}(H)$, therefore containing a
  neighborhood (in $J(R)$) of $x$.  Also,
  
  \begin{align*}
    H^{-1} \cap A^{-1} &\subset R^{-1}(H) \cap R^{-1}(A) \\
    &= R^{-1}(H \cap A) \\
    &= \emptyset.
  \end{align*}
  Thus $x \notin T(A^{-1})$.  Since $x$ was any point in $X \setminus
  A^{-1}$, we see that $T(A^{-1})=A^{-1}$.
\end{proof}

\begin{thm}\label{thm:elements_to_elements}
  Let $R$ be a rational function.  If the Julia set $J$ is a finished
  continuum, then for each $D \in \mathcal{D}(J(R))$ we have $R(D) \in
  \mathcal{D}(J(R))$.
\end{thm}

\begin{proof}
  Let $Y = J(R) / \mathcal D(J(R))$, and let $m:J(R) \rightarrow Y$ be
  the associated quotient map.  Recall from \prettyref{thm:decomp}
  that elements of $\mathcal D(J(R))$ are nowhere dense in $J(R)$ and
  that $Y$ is a finite tree.  For $x \in J(R)$, recall that $D(x)$
  denotes the unique element of $\mathcal D(J(R))$ containing $x$.  We
  must show that $R(D(x)) = D(R(x))$.  Let $K$ be the component of
  $R^{-1}(D(R(x)))$ containing $x$.  Because $R$ is confluent, we see
  that $R(K) = D(R(x))$.  It is then sufficient to show that $D(x) =
  K$.
  
  We see that $K$ has empty interior in $J(R)$, since $R(K) = D(R(x))$
  has empty interior in $J(R)$ and $R|_{J(R)}$ is an open map.  As a
  component of a $T$-closed set, $K$ is $T$-closed \cite[Lemma
  2.6]{FitzGerald:1967rz} and the union of elements of $\mathcal
  D(J(R))$ by \prettyref{lem:closed_saturated}.  Therefore, $K =
  m^{-1}(m(K))$, so $m(K)$ must have empty interior in $Y$.  However,
  $Y$ is a finite tree, so the only subcontinua with empty interior
  are points.  We then conclude that $m(K)$ is a point, and therefore
  $K$ is an element of $\mathcal{D}(J(R))$.
\end{proof}

Now we study the induced map $g:J(R)/\mathcal{D}(J(R)) \rightarrow
J(R)/\mathcal{D}(J(R))$ and use it to draw conclusions about the set
$J(R)$ itself.

\begin{thm}\label{cor:induced_map}
  Suppose that $R$ is a rational map whose Julia set is a finished
  continuum.  Then the monotone map $m:J(R) \rightarrow
  J(R)/\mathcal{D}(J(R))$ is a monotone semiconjugacy of $R|_{J(R)}$
  to a map $g:J(R)/\mathcal{D}(J(R)) \rightarrow
  J(R)/\mathcal{D}(J(R))$ of a finite tree.  The map $g$ is then open,
  topologically exact, and at most $\operatorname{deg}(R)$-to-one.
\end{thm}
\begin{proof}
  For brevity, set $Y = J(R) / \mathcal{D}(J(R))$.  Notice that $m
  \circ R$ is constant on every set of the form $m^{-1}(y)$ by
  \prettyref{thm:elements_to_elements}.  Therefore, the map $g:Y
  \rightarrow Y$ defined by $g = m \circ R \circ m^{-1}$ is a
  single-valued, continuous map (see \cite[Theorem
  22.2]{Munkres:1975ad}).

  To see that $g$ is open and topologically exact, let $U \subset Y$
  be open.  Then $m^{-1}(U)\subset J(R)$ is an open set, saturated
  with respect to the quotient map $m$.  Since $R|_{J(R)}$ is open,
  $R(m^{-1}(U))\subset J(R)$ is open, and by
  \prettyref{thm:elements_to_elements} this set is also saturated.
  Therefore, $m(R(m^{-1}(U)))=g(U) \subset Y$ is open, so $g$ is an
  open map.  Also, since $m^{-1}(U) \subset J(R)$ is open, there
  exists $n \in \nats$ such that $R^n(m^{-1}(U))=J(R)$.  Thus $m \circ
  R^n \circ m^{-1}(U)=g^n(U) = Y$, so $g$ is topologically exact.
	
  To show that $g$ is at most $\operatorname{deg}(R)$-to-one, let $y
  \in Y$. Then $R^{-1}(m^{-1}(y))$ is the union of at most
  $\operatorname{deg}(R)$ elements of $\mathcal{D}(J(R))$, since each
  element must map onto $m^{-1}(y)$ and hence must contain a preimage
  of $y$.  Therefore, $m(R^{-1}(m^{-1}(y)))=g^{-1}(y)$ has cardinality
  at most $\operatorname{deg}(R)$.
\end{proof}

\begin{cor}\label{cor:arc}
  If $R$ is a rational function and $J(R)$ is a finished continuum,
  then $J(R) / \mathcal D(J(R))$ is an arc, and $J(R)$ is irreducible.
\end{cor}

\begin{proof}
  Suppose that $J(R) / \mathcal{D}(J(R))$ has a cut point $x$ of order
  at least $3$.  Since $g$ is topologically exact, we have that $B =
  \bigcup_{n \in \nats}g^{-n}(x)$ is dense in
  $J(R)/\mathcal{D}(J(R))$.  Since $g$ is open, each point of $B$ is
  also a branch point of $J(R) / \mathcal{D}(J(R))$.  Therefore, $J(R)
  / \mathcal{D}(J(R))$ has infinitely branch points if it has one.  A
  finite tree does not have infinitely many branch points, so $J(R) /
  \mathcal{D}(J(R))$ is an arc.  That $J(R)$ is irreducible follows
  from \prettyref{thm:decomp} and that the arc is an irreducible
  continuum.
\end{proof}

\begin{rem}\label{rem:nsaw}
  It is known that topologically exact open maps of intervals are
  conjugate to $n$-saw maps, which are of the form
  \[
  f(x)=\begin{cases}
    (n-2i)x & \mbox{if $x \in [2i, 2i+1]$}\\
    (2(i+1)-n)x & \mbox{if $x \in [2i+1, 2i+2]$}.
  \end{cases}
  \]
  Therefore, $g$ is conjugate to a map of this sort.
\end{rem}

Combining the above, we can prove the main theorem of the section.

\begin{proof}[Proof of \prettyref{thm:main_rat}]
  Suppose that $J$ contains no indecomposable subcontinuum with
  interior.  The monotone map $m$ is the quotient map corresponding to
  the decomposition of \prettyref{thm:decomp}.  That $m$ is an arc is
  \prettyref{cor:arc}, and the facts about the induced map $g:I
  \rightarrow I$ are from \prettyref{cor:induced_map}.
\end{proof}

\section{The Polynomial Case}
\label{sec:polynomial-case}

In the case of polynomial Julia sets, we can say more.
\begin{thm}
  Let $P$ be a polynomial, and suppose that $J(P)$ is finitely
  irreducible.  Then either
  \begin{enumerate}
  \item $J$ is an indecomposable continuum, or
  \item $J$ is homeomorphic to an arc.
  \end{enumerate}
\end{thm}

\begin{proof}
  Suppose that $J$ is not indecomposable.  Then no indecomposable
  subcontinuum has interior in $J$ \cite[Theorem 1]{Childers:2006fk},
  so $J(P)$ is a finished continuum.  Let $m:J \rightarrow [0,1]$ be
  the monotone map provided by \prettyref{thm:main_rat}, and let $g:I
  \rightarrow I$ be the map to which $P|_{J(P)}$ is conjugate via $m$.
  We will show that $J(P)$ is homeomorphic to an arc by showing it is
  locally connected (see \prettyref{rem:lc_irred}).  This will be
  accomplished by showing that the forward orbits of its critical
  points are finite \cite[Theorem 19.7]{Milnor:2006fr}.

  For $t \in [0,1]$, let $J_t$ denote the union of the continuum
  $m^{-1}(t)$ with its bounded complementary domains.  It is not
  difficult to see that $P(J_t)=J_{g(t)}$, and that the map
  $P|_{J(t)}:J_t \rightarrow J_{g(t)}$ is open for each $t \in I$ (for
  instance by \cite[Lemma 13.13]{Nadler:1992ek}).

  Extend $m:J(P) \rightarrow [0,1]$ to a map $\hat m:\bigcup_{t \in
    [0,1]} J_t \rightarrow [0,1]$ by sending points of $J_t$ to the
  point $t \in [0,1]$.  Each $\hat m^{-1}(t)$ is a non-separating
  plane continuum, so the Vietoris-Begle theorem implies that $\hat m$
  induces an isomorphism between the \v{C}ech cohomologies of
  $\bigcup_{t \in [0,1]} J_t$ and $[0,1]$.  Therefore, $\bigcup_{t \in
    [0,1]}J_t$ is a non-separating plane continuum containing $J(P)$,
  so $\bigcup_{t \in [0,1]} J_t$ equals the filled Julia set.  Since
  $J(P)$ is connected, each critical point is contained in the filled
  Julia set \cite[9.5]{Milnor:2006fr} and therefore in some $J_t$.

  We now show that a point $t \in [0,1]$ is a critical point of $g$ if
  and only if $J_t$ contains a critical point of $P$.  Suppose that a
  fiber $J_t$ does not contain any critical point of $P$.  Then there
  is a (saturated) neighborhood $U$ of $J_t$ on which $P$ is a
  homeomorphism.  In particular, no two fibers contained in $U$ have
  the same image, so $g$ is a homeomorphism on $m(U)$.

  On the other hand, suppose that $t$ is not a critical point of $g$.
  Then there is a (saturated) neighborhood $U$ of $J_t$ such that no
  two fibers contained in $U$ have the same image.  If any fiber
  $J_{t_0}$ does not map by a homeomorphism onto its image, then
  $J_{t_0}$ must contain a critical point since it and its image are
  full continua.  Let $V \subset U$ be a disk about $J_t$ such that
  $R|_{V \setminus J_t}$ is a covering map onto its image.  Then a
  fiber which does not contain a critical point intersects $V$, and
  $R$ is one-to-one in a neighborhood of it.  Therefore, $R|_{V}$ is a
  homeomorphism, and $J_t$ cannot contain a critical point.

  If one takes a saturated open set $V \subset U$ about $J_t$ such
  that $V \setminus J_{t_0}$ does not contain the preimage of any
  critical value, we see that $R|_V$ is a homeomorphism, so $J_t$
  contains no critical point of $R$.

  Notice that $J_0 \cup J_1$ is a forward-invariant set.  Suppose
  without loss of generality that $P(J_0) = J_0$.  Note that $J_0$
  cannot contain a critical point of $P$, since $0$ cannot be a
  critical point of $g$.  Therefore, $P|_{J_0}$ is a forward expanding
  homeomorphism, which implies that $J_0$ must be a point \cite[Lemma
  18.8]{Milnor:2006fr}.  Then $J_1$ must also be a point, either
  because it is fixed or because it maps to $J_0$.  Because every
  critical point of $g$ must map into $\set{0,1}$, every critical
  point of $R$ must map into $J_0 \cup J_1$ and hence has a finite
  orbit.  This implies that $J(P)$ is locally connected, and therefore
  an arc.
\end{proof}

\section{Further Work}
The main novelty in this work is the use of what might be called a
\emph{locally connected model} of the Julia set to draw conclusions
about its topology and the dynamics of its rational map.  There is
more structure on polynomial Julia sets than there are on rational
Julia sets, so it is to be expected that more can be concluded in the
polynomial case.  Locally connected models for polynomial Julia sets
have been studied further in joint work with Alexander Blokh and Lex
Oversteegen in \cite{Blokh:2008il}, where we characterized the
\emph{finest locally connected model} for the action of a polynomial
on its Julia set.

Vought's decomposition is rather special, and only applies to finitely
irreducible continua.  It is an example of a broader notion called a
\emph{core decomposition}.  A decomposition is core with respect to a
property P if it has property P and it refines all decompositions with
property P.  Therefore, Vought's decomposition for finished continua
is core with respect to the property that the quotient space is
locally connected.

Generally speaking, there is no core decomposition of an arbitrary
continuum with locally connected quotient.  However, there is always a
core decomposition with an aposyndetic
quotient \cite{FitzGerald:1967rz}, of which Vought's decomposition is
a special case. Such a model may serve in the stead of a locally
connected model where locally connected models are unavailable.
However, the utility is limited when the continuum in question
contains an indecomposable subcontinuum with interior, as it must be
absorbed into an element of any decomposition to a nice space.  There
are several pertinent questions.

\begin{ques}
  Does there exist a rational function whose Julia set contains an
  indecomposable continuum with interior?
\end{ques}
\begin{ques}
  Does there exist a rational function whose Julia set does not have a
  finest locally connected model?
\end{ques}
\begin{ques}
  Let $R$ be a rational function with connected Julia set.  Is the
  finest decomposition to an aposyndetic continuum invariant with
  respect to $R$?
\end{ques}
\begin{ques}
  For what useful topological properties $P$ does there exist a finest
  decomposition of every Julia set $J(R)$ satisfying $P$?  Is the
  decomposition dynamic?  Which of these is the appropriate
  analog for the finest locally connected model?
\end{ques}


\bibliographystyle{annotate}

\end{document}